\documentclass{amsproc}
\usepackage{amsmath, amsthm, amssymb, amscd, latexsym}
\usepackage{hyperref}
\usepackage{graphicx}
\usepackage[all]{xy}
\theoremstyle{plain}
 \newtheorem{thm}{Theorem}
 \newtheorem{prop}{Proposition}
 
 \newtheorem{cor}{Corollary}
\theoremstyle{definition}
 \newtheorem{dfn}{Definition}
 \newtheorem{rem}{Remark}

\title{Higher ramification loci over homogeneous spaces}
\author{Yu-Chao Tu}
\thanks{
Department of Mathematics, University of Utah, Salt Lake City, UT 84102, U.S.A.}
\thanks{
Email: tu@math.utah.edu
}

\thanks{
Mathematics Subject Classification (2010): 14E22, 14L40
}

\thanks{
The author is partially supported by the Focused Research Grant(FRG).
}

\begin{document}
\begin{abstract}
We generalize the Gaffney--Lazarsfeld theorem on higher ramification loci of branched coverings of $\mathbb{P}^n$ to  homogeneous spaces with Picard number one.
\end{abstract}
\maketitle

\section{Introduction}

Let $f:X\to Y$ be a covering map, i.e., a finite surjective morphism, between irreducible projective varieties over $\mathbb C$ with $Y$ smooth.
A local measure of the ramification of $f$ around a point $x$ of $X$ is the \emph{local degree} $e_f(x)$ defined by Gaffney and Lazarsfeld in \cite{GL80}. The map $f$ is ramified at $x$ precisely when $e_f(x)>1$ and intuitively, the local degree around $x$ counts the number of sheets coming together near $x$. It is natural to consider the higher ramification loci
\[ R_{\ell}(f):= \{x \in X \mid e_f(x) > \ell \}.
\]
\noindent For example, $R_1(f)$ is the classical ramification locus of $f$. When $X$ is normal, a result of Zariski (\cite[Proposition 2]{Zar58}) says that the ramification locus of $f$ has pure codimension one or is empty.

Our main result offers a Zariski-type conclusion for the higher ramification loci over homogeneous spaces of Picard number one.

\begin{thm}\label{thm1}
Let $f:X\to Y$ be a covering map of degree $d > 1$ between irreducible projective varieties of dimension $n$ over $\mathbb C$. Assume that $Y=G/P$ is a projective homogeneous space,
where $G$ is a connected linear algebraic group, and assume that $Y$ has Picard number one.
If $\ell\leq\min\{d-1, n \}$,  $R_{\ell}(f)$ is nonempty and
\[
\mathrm{codim}_X R_{\ell}(f) \leq \ell.
\]
\end{thm}
For a list of homogeneous spaces with Picard number one, see Mok's paper \cite[page 8]{Mok08}.

The case $\ell=1$ is Zariski's theorem: we get that the ramification locus is a hypersurface. Also note that the case $Y=\mathbb P^n$ is a theorem due to Gaffney and Lazarsfeld (see \cite[Theorem 1]{GL80}).

On the other hand, Debarre proves a theorem similar to Theorem \ref{thm1} when $Y$ is a simple abelian variety and $f: X \rightarrow Y$ does not factor through nontrivial isogenies \cite[Theorem 7.1]{Deb95}\cite[Example 3.5.8]{Laz04}.


The proof of the Gaffney--Lazarsfeld theorem in \cite[page 57]{GL80} uses a Bertini theorem on the irreducibility of general hyperplane sections \cite[Theorem 3.3.1]{Laz04}, which allows induction on the dimension of $\mathbb P^n$,
and the Fulton--Hansen connectedness theorem \cite[Theorem 3.1]{FL81}. We explain some of the problems with generalizing this approach to other homogeneous spaces.




To run the induction on dimension, one would like to exhibit a suitable subvariety of $Y$ that is ``similar'' to $Y$. A step in this direction is \cite[Th\'eor\`eme 6.2]{Deb96}, where Debarre shows that if $Y$ is a Grassmannian, there are natural subvarieties that are also Grassmannians, and such that their pull-backs are irreducible.
Unfortunately, these are not hypersurfaces, which leads to gaps in the induction argument.


The other tool used in \cite{GL80} is the Fulton--Hansen connectedness theorem:
\begin{thm}\label{FH}(Fulton--Hansen connectedness theorem)\cite[Theorem 3.1]{FL81}
Let $X$ be an irreducible projective variety and let $f: X \rightarrow \mathbb P^n \times \mathbb P^n$ be a morphism.
Assume that $\dim f(X) > n$.
Then the inverse image $f^{-1}(\Delta) \subseteq X$ of the diagonal $\Delta \subset \mathbb P^n \times \mathbb P^n$ is connected.
\end{thm}

Various other connectedness results have been proved when $Y$ is replaced with a Grassmannian (\cite{Han83}, \cite{Deb96}), with a product of projective spaces (\cite{Deb96}), with a $G/P$ (\cite[Korollar, page 148]{Fal81}), with a weighted projective space (\cite[Theorem 0.1]{Bad96}), or with a product of weighted projective spaces (\cite[Theorem 2.1]{BR08}). Their conclusions however are not strong enough for our purposes.

Therefore, it is unclear how one would prove Theorem \ref{thm1} by an argument paralleling the proof of the Gaffney--Lazarsfeld theorem: we have no good candidate for a subvariety of $G/P$ that we could use for induction, and the known connectedness results are too weak. Instead of running induction on the dimension of $Y$, we use a result of Koll\'ar (\cite[Theorem 2]{Kol13}) facilitating induction on the index $\ell$ of the higher ramification loci.

Finally, recall that if $f:X\to Y$ is a covering map of degree $d$ of smooth projective varieties over $\mathbb C$, the natural map $\mathcal O_Y\to f_*\mathcal O_X$ splits via the trace morphism, so that $f_*\mathcal O_X=\mathcal O_Y\oplus F$ for some vector bundle $F$ of rank $d-1$ on $Y$.
The \emph{associated vector bundle} of the covering is the dual bundle $E:=F^*$ (see \cite[\S6.3.D]{Laz04}).
If $E$ is ample, the conclusion of Theorem \ref{thm1} holds (see \cite[Example 6.3.56]{Laz04}).
However, the ampleness of $E$ is only known in two cases: projective spaces by Lazarsfeld \cite[Proposition 1.2]{Laz80-2} and Lagrangian Grassmannians of maximal isotropic subspaces of symplectic vector spaces by Kim and Manivel \cite[Theorem 2]{KM98}, but not when $Y$ is a general rational homogeneous space.
When $Y$ is a simple abelian variety and $f: X \rightarrow Y$ does not factor through nontrivial isogenies, the ampleness of $E$ is also known \cite[Theorem 7.1]{Deb95}.





As an application of Theorem \ref{thm1}, we show in Corollary \ref{sc} that an irreducible normal projective variety which is a finite cover of small degree of a rational homogeneous space with Picard number one is simply connected.

\section{Preliminaries}
We assume that all varieties are over $\mathbb{C}$.

A \emph{covering map} is a finite surjective morphism $f: X \rightarrow Y$ between irreducible varieties. In this paper, $Y$ is always smooth.

\begin{dfn}\label{dfn1}\cite[page 54]{GL80}\cite[Definition 3.4.7]{Laz04}
Let $f: X \rightarrow Y$ be a covering map with $Y$ smooth.
For $x \in X$, let $U(f(x))$ be a small neighborhood about $f(x)$ (in the classical topology) and let $V(x)$ be the connected component of $f^{-1}(U(f(x)))$ containing $x$.
The \emph{local degree} of $f$ at $x$ is then
\[
e_f(x):= \deg (V(x) \rightarrow U(f(x))).
\]
Note that
\[
f^{-1}(U(f(x))) = \mathop\coprod_{f(x')=f(x)} V(x').
\]
The $V(x')$ are disjoint connected neighborhoods of points in $f^{-1}(f(x))$.
\end{dfn}

From a geometric perspective, if a point has higher local degree, there are more sheets coming together near this point.
We have the basic formula \cite[(1.4)]{GL80}\cite[(3.11), page 216]{Laz04}
\begin{equation}\label{eq:multsum}
\mathop{\sum}_{f(x')=f(x)} e_f(x') = \deg f.
\end{equation}

\begin{dfn}\label{dfn2}\cite[page 217]{Laz04}
Assume the same set-up as in Definition \ref{dfn1}.
The \emph{higher ramification loci} of $f$ are
\[
R_{\ell}(f):= \{x \in X \mid e_f(x) > \ell \}.
\]
\end{dfn}
\noindent Note that $R_1(f)$ is the usual ramification locus of $f$. In general, $R_{\ell}(f)$ is a closed algebraic subset of $X$.

We also use the following result in Koll\'ar's paper \cite[Theorem 2.12]{Kol13}.

\begin{thm}\cite[Theorem 2.(12)]{Kol13}\label{thm3}
Let $Y = G/P$ be a simply connected projective homogeneous space with Picard number one.
For any positive-dimensional irreducible subvariety $Z \subset Y$,
any nonempty Zariski open subset $Y^0 \subset Y$,
and any finite cover $u_Z: \tilde Z \rightarrow Z$, the induced map $\pi_1((\tau_g \circ u_Z)^{-1}(Y^0)) \rightarrow \pi_1(Y^0)$ is surjective for general $g \in G$.
\end{thm}

In the theorem, $\tau_g: Y \rightarrow Y$ is the left translation by the group element $g \in G$.

\begin{rem}\label{rem1}
Koll\'ar's result also implies the following.
For any irreducible subvariety $Z \subset Y \times Y$ such that
neither of the two projections $Z \rightarrow Y$ is constant,
any nonempty Zariski open subset $Y^0 \subset Y \times Y$,
and any finite cover $u_Z: \tilde Z \rightarrow Z$,
the induced map $\pi_1((\tau_{(g_1, g_2)} \circ u_Z)^{-1}(Y^0)) \rightarrow \pi_1(Y \times Y)$ is surjective for general $(g_1, g_2) \in G \times G$.
\end{rem}

The key point here is that every positive-dimensional algebraic subvariety of a projective homogeneous space with Picard number one is nondegenerate and so is any subvariety $Z$ as in Remark \ref{rem1}.


\section{Proof of Theorem \ref{thm1}}

In this section we prove Theorem \ref{thm1}, and Corollary \ref{sc} states that normal covers of small degree of rational homogeneous spaces of Picard number one are simply connected.


\begin{proof}[Proof of Theorem \ref{thm1}]
First, we reduce the case to $X$ normal.
Let $\bar{n}: \bar{X} \rightarrow X$ be the normalization of $X$. Then $f\circ\bar{n}: \bar{X} \rightarrow Y$ is also a covering map.
Assume that Theorem \ref{thm1} holds for $f\circ\bar{n}: \bar{X} \rightarrow Y$,
so $R_k(f\circ\bar{n})$ is nonempty and $\mathrm{codim}_X R_k(f\circ\bar{n}) \leq k$.
Therefore, we have $\mathrm{codim}_X \bar{n}(R_k(f\circ\bar{n})) \leq k$.
Since $R_k(f)\supseteq\bar{n}(R_k(f\circ\bar{n}))$, we obtain $\mathrm{codim}_X R_k(f) \leq k$.
So we can assume $X$ normal in the rest of the proof.

We use induction on $\ell$.
For $\ell=1$, Zariski's purity theorem gives the result.
Now assume that Theorem \ref{thm1} holds for all $\ell \leq k-1$.
By the following result of Lazarsfeld, it suffices to show that $R_k(f)$ is nonempty.

\begin{prop}\label{prop1}
\cite[Theorem 2.2]{Laz80}
Let $f: X \rightarrow Y$ be a covering map between irreducible varieties with $X$ normal and $Y$ nonsingular. If $R_{\ell}(f)$ is nonempty,
for every irreducible component $S$ of $R_{\ell}(f)$, we have $\mathrm{codim}_XS \leq \ell$.
\end{prop}

\begin{proof}[Proof of Proposition 1]
The proof is in the thesis \cite[Theorem 2.2]{Laz80} and it is similar to \cite[Lemma 2.4, Remark 2.5]{Gaf88}.
The proof is by induction on $\ell$.
For $\ell = 0$, we have $R_0(f) = X$.
Assume that the result is true for all $\ell < k$ and that $R_k(f)$ is nonempty, and choose $x \in R_k(f)$.
Let $V \subset R_{k-1}(f)$ be an irreducible component that contains $x$.
By the induction hypothesis, we have $\dim V \geq n-k+1$.
If $V$ is also an irreducible component of $R_k(f)$, we are done.
So we assume that $V$ is not an irreducible component of $R_k(f)$.
Since $f$ is finite surjective and $Y$ is smooth, each fiber has $\deg f$ points counting multiplicities by \eqref{eq:multsum}.
Then there exist $p \in V, q \in X, p \neq q$, and $f(p) = f(q)$.

Consider the map $F:= f \times (f\mid_{V}): X \times V \rightarrow Y \times Y$, and let $\Delta_Y \subset Y \times Y$ denote the diagonal.
From the existence of $p$ and $q$, we deduce the existence of a component $T \subset F^{-1}(\Delta_Y)$ such that $T \neq V \times V$ and $T \cap (V \times V) \neq \emptyset$.
We have $p_1(T \cap (V\times V)) \subset R_k(f)$, where $p_1$ is the first projection.
By \cite[Lemma 2.4]{Gaf88}, the set $F^{-1}(\Delta_Y)$ is connected in dimension $n-k$, therefore $\dim\, R_k(f) \geq n-k$, which is the desired result.
\end{proof}

\begin{rem}
Recall that a Noetherian scheme $X$ is connected in dimension $k$ if $\dim X > k$ and $X \backslash T$ is connected for every closed subset $T \subset X$ with $\dim T < k$.
For a brief list of properties of $k$-connectedness, see \cite[Section 3.3.C]{Laz04}.
\end{rem}

We now show that $R_k(f)$ is nonempty.
By the induction hypothesis,
we can choose an irreducible component $S$ of $R_{k-1}(f)$ with $\dim S \geq n-k+1 \geq 1$,
so that $f(S)$ is a positive-dimensional subvariety in $Y$, and we can apply Theorem \ref{thm3} to $f(S)$.

Take $Y^0 = Y\backslash f(R_1(f))$ and $X^0 = f^{-1}(Y^0)$.
Then $Y^0$ is a nonempty Zariski open subset of $Y$ since ${\rm codim}_X R_1(f)\geq 1$.
Let $g \in G$ be a general element, so that $(\tau_g\circ f(S)) \cap Y^0 \neq \emptyset$.
Set $S' := (\tau_g\circ f|_S)^{-1}(Y^0)$.
It is a nonempty open subset of $S$ and $S$ is irreducible, so $S'$ is irreducible.
We have a diagram
\[
\xymatrix{
X^0 \mathop{\times}_{Y^0} S' \ar[r] \ar[d] & S' \ar[d] \ar@{^{(}->}[r] & S \ar[dd]^{\tau_g\circ f} \\
X^0 \ar[r]^{\mathrm{\acute{e}tale}} \ar@{^{(}->}[d] & Y^0 \ar@{^{(}->}[dr] & \\
X \ar[rr]^f & & Y }
\]
By Theorem \ref{thm3}, the map
\[
\pi_1(S') \rightarrow \pi_1(Y^0)
\]
is surjective, hence $X^0 \mathop{\times}_{Y^0} S'$ is connected.
Note that $X^0 \mathop{\times}_{Y^0} S'$ is a nonempty open subset of $X \times_Y S_g$,
where $S_g\rightarrow Y$ is the twisted morphism $\tau_g\circ f|_S: S \rightarrow Y$.
By \cite[Theorem 2]{Kle74} and since the rational points of $G$ are dense in $G$, we have that for general $g \in G$, the fiber product $X \times_Y S_g$ is equidimensional.
Since $f$ is finite and surjective, each irreducible component $\Sigma$ of $X \times_Y S_g$ dominates $S$,
so the image of the projection from $\Sigma$ to $S$ intersects $S'$,
which implies that $\Sigma \cap X^0 \mathop{\times}_{Y^0} S' \neq \emptyset$,
and then $X^0 \mathop{\times}_{Y^0} S'$ is dense in $X \times_Y S_g$.
Since $X^0 \mathop{\times}_{Y^0} S'$ is connected,
$X \times_Y S_g$ is connected for $g$ general \cite[Chapter 1.4, Problem 1]{Bre93}.
We claim that $X \times_Y S$ is connected.
Set
\[
Z := \{(x, s, g) \in X \times S \times G \ | \ g \cdot f(s) = f(x) \}.
\]
Then $Z$ is irreducible because the fibers of the projection $Z \rightarrow X \times S$ are all isomorphic to $P$.
The fibers of the projection $Z \rightarrow G$ above $g$ is $X \times_Y S_g$;
it is connected for $g$ general, hence connected for all $g$ by Zariski's Main Theorem, because $G$ is normal.

If $X \mathop{\times}_{Y} S = \Delta_S = \{(s, s) \in X\times X \mid s \in S \}$,
which means that $f$ is one-to-one over points of $f(S)$,
we have $e_f(s) = d \geq k+1$ for $s \in S$ and the result is true.
Now we assume that $X \mathop{\times}_{Y} S \neq \Delta_S$.
We know that $X \mathop{\times}_{Y} S$ is connected, so there is an irreducible component $T$ of
$X \mathop{\times}_{Y} S$ such that
$T \neq \Delta_S$ and $T \cap \Delta_S \neq \emptyset$.
We can pick a sequence $(t_h, s_h) \in T$ such that
$t_h \neq s_h$ and $f(t_h)=f(s_h)$ for
every $h \geq 1$, and
\[\lim_{h \rightarrow \infty} t_h = \lim_{h \rightarrow \infty} s_h =: s^*.
\]
We have $e_f(t_h) \geq 1, e_f(s_h) \geq k$, hence
\[
e_f(s^*) \geq k+1 \ \text{(this is in \cite[(3.12)]{Laz04})}
\]
which gives $s^* \in R_k(f) \neq \emptyset$.
So $R_k(f)$ is nonempty and the proof is complete.
\end{proof}



We can prove the following corollary of Theorem \ref{thm1}.
\begin{cor}\label{sc}
Let $f: X \rightarrow Y$ be a covering map from an irreducible normal Cohen-Macaulay projective variety $X$ to a homogeneous space $Y$ with Picard number one.
If $\deg f \leq \dim X$, $X$ is simply connected.
\end{cor}

\begin{proof}
The algebraic simple connectedness of $X$ follows from Theorem \ref{thm1} as in \cite[Corollary 3.4.10]{Laz04}.
As in \cite[Example 3.4.11]{Laz04}, by Theorem \ref{thm1} applied with $\ell = \deg f - 1$,
there exists an irreducible curve $T \subset X$ such that $f$ is one-to-one over $f(T)$.
Let $n: \bar{T} \rightarrow T$ be the normalization and consider the map
\[
F:= (f\mid_{T}\circ\ n, f\mid_{T}\circ\ n): \bar{T} \rightarrow Y \times Y.
\]
Note that
$(X \times \bar{T}) \times_{Y \times Y} \bar{T}$ is homeomorphic to $\bar{T}$.
We claim that the map
\[
\pi_1(\bar{T}) = \pi_1((X \times \bar{T}) \times_{Y \times Y} \bar{T}) \rightarrow
\pi_1(X \times \bar{T})
\]
is surjective. 
To see this, note that $Y \times Y$ is homogeneous and Remark \ref{rem1} applies to $F(\bar{T})$.
For a general element $(g_1, g_2) \in G\times G$, where $Y=G/P$, we have $(\tau_{(g_1, g_2)}\circ F(\bar{T})) \cap (Y^0 \times Y^0) \neq \emptyset$, where $Y^0 \times Y^0$ is the unramified locus of $(f, f): X \times X \rightarrow Y \times Y$.
Setting $\bar{T}^0:= (\tau_{(g_1, g_2)}\circ F)^{-1}(Y^0 \times Y^0) \subset \bar{T}$ and $X^0:= f^{-1}(Y^0)$, we have the diagram
\[
\xymatrix{
(X^0 \times \bar{T}^0) \times_{Y^0 \times Y^0} \bar{T}^0 \ar[r] \ar[d] & \bar{T}^0 \ar[d] \ar@{^{(}->}[r] & \bar{T} \ar[dd]^{\tau_{(g_1, g_2)}\circ F} \\
X^0 \times \bar{T}^0 \ar[r]^{\mathrm{\acute{e}tale}} \ar@{^{(}->}[d] & Y^0 \times Y^0 \ar@{^{(}->}[dr] & \\
X \times \bar{T} \ar[rr] & & Y \times Y. }
\]
By Remark \ref{rem1}, $\pi_1(\bar{T}^0) \rightarrow \pi_1(Y^0 \times Y^0)$ is surjective.
Consider the diagram
\[
\xymatrix{
 & \pi_1(\bar{T}) \ar@{=}[d] \\
\pi_1((X^0 \times \bar{T}^0) \times_{Y^0 \times Y^0} \bar{T}^0) \ar[r]\ar[d] & \pi_1((X \times \bar{T}) \times_{Y \times Y} \bar{T}) \ar[d] \\
\pi_1(X^0 \times \bar{T}^0) \ar[r] &  \pi_1(X \times \bar{T}).
}
\]
The left vertical map is surjective since $\pi_1(\bar{T}^0) \rightarrow \pi_1(Y^0 \times Y^0)$ is surjective by Remark \ref{rem1} and $X^0 \times \bar{T}^0 \rightarrow Y^0 \times Y^0$ is \'etale.
The two horizontal maps are also surjective by the irreducibility of
$(X \times \bar{T}) \times_{Y \times Y} \bar{T}$ and of $X \times \bar{T}$.
Therefore, the map $\pi_1(\bar{T}) \rightarrow \pi_1(X \times \bar{T})$ is surjective, which implies $\pi_1(X) = 1$, so $X$ is simply connected.
\end{proof}



\section*{Acknowledgments}

The author would like to express his deepest gratitude to Professor J\'anos Koll\'ar for his generous and continuous support, and for some enlightening discussions. The author also offers special thanks to Professor Robert Lazarsfeld for his valuable suggestions and references including the proposition in section 3 appearing in his thesis \cite[Theorem 2.2]{Laz80}. The author wants to thank Doctor Aurel Mihai Fulger for many helpful suggestions.
The author is grateful to the referee for many valuable suggestions and corrections about this paper.
The discussions with Professor Tommaso de Fernex and Professor Karl Schwede also helped the author very much.

\nocite{Laz04}
\bibliographystyle{plain}
\bibliography{bibtex_YC_Tu_1}

\begin{thebibliography}{10}

\bibitem{Bre93}
Glen~E. Bredon.
\newblock {\em Topology and geometry}.
\newblock Springer, 1993.

\bibitem{Bad96}
Lucian B\u{a}descu.
\newblock Algebraic {B}arth-{L}efschetz theorems.
\newblock {\em Nagoya Math. J.}, 142:17--38, 1996.

\bibitem{BR08}
Lucian B\u{a}descu and Flavia Repetto.
\newblock A connectedness theorem for products of weighted projective spaces.
\newblock {\em Comm. in Alg.}, 36:2958--2968, 2008.

\bibitem{Deb95}
Olivier Debarre.
\newblock Th\'eor\`emes de connexit\'e et vari\'et\'es ab\'eliennes.
\newblock {\em Amer. J. Math.}, 117:787--805, 1995.

\bibitem{Deb96}
Olivier Debarre.
\newblock Th\'eor\`emes de connexit\'e pour les produits d'espaces projectifs
  et les grassmanniennes.
\newblock {\em Amer. J. Math.}, 118, no. 6:1347--1367, 1996.

\bibitem{Fal81}
Gerd Faltings.
\newblock Formale {G}eometrie und homogene {R}\"{a}ume.
\newblock {\em Invent. Math.}, 64:123--165, 1981.

\bibitem{FL81}
William Fulton and Robert Lazarsfeld.
\newblock Connectivity and its applications in algebraic geometry.
\newblock {\em Lect. Notes in Math.}, 862:26--92, 1981.

\bibitem{Gaf88}
Terence Gaffney.
\newblock Multiple points, chaining and {Hilbert} schemes.
\newblock {\em Amer. J. Math.}, 110:595--628, 1988.

\bibitem{GL80}
Terence Gaffney and Robert Lazarsfeld.
\newblock On the ramification of branched coverings of $\mathbb{P}^n$.
\newblock {\em Invent. Math.}, 59:53--58, 1980.

\bibitem{Han83}
Johan Hansen.
\newblock A connectedness theorem for flagmanifolds and {G}rassmannians.
\newblock {\em Amer. J. Math.}, 105:633--639, 1983.

\bibitem{KM98}
Meeyoung Kim and Laurent Manivel.
\newblock On branched coverings of some homogeneous spaces.
\newblock {\em Topology}, 38:1141--1160, 1999.

\bibitem{Kle74}
Steven~L. Kleiman.
\newblock The transversality of a general translate.
\newblock {\em Compos. Math.}, 28:287--297, 1974.

\bibitem{Kol13}
J\'anos Koll\'ar.
\newblock Neighborhoods of subvarieties in homogeneous spaces.
\newblock {\em arXiv:1308.5603v1}, 2013.

\bibitem{Laz80-2}
Robert Lazarsfeld.
\newblock A {B}arth-type theorem for branched coverings of projective space.
\newblock {\em Math. Ann.}, 249:153--162, 1980.

\bibitem{Laz80}
Robert Lazarsfeld.
\newblock Branched coverings of projective space, {P}h.{D} dissertation.
\newblock 1980.

\bibitem{Laz04}
Robert Lazarsfeld.
\newblock {\em Positivity in algebraic geometry I, II}.
\newblock Springer-Verlag, 2004.

\bibitem{Mok08}
Ngaiming Mok.
\newblock Recognizing certain rational homogeneous manifolds of {P}icard number
  1 from their varieties of minimal rational tangents.
\newblock {\em Third International Congress of Chinese Mathematicians. Part 1,
  2, AMS/IP Stud. Adv. Math., 42, pt. 1, vol. 2, Amer. Math. Soc., Providence,
  RI}, 41-61, 2008.

\bibitem{Zar58}
Oscar Zariski.
\newblock On the purity of the branch locus of algebraic functions.
\newblock {\em Proc. Natl. Acad. Sci. U.S.A.}, 44:791--796, 1958.

\end{thebibliography}

\end{document}